\documentclass[12pt]{article}
\usepackage{amssymb,latexsym,amsmath,eufrak,amsthm,cite,pictex}
\usepackage[active]{srcltx}
\usepackage{hyperref}
\usepackage[inline]{enumitem}
\usepackage[T1]{fontenc}
\usepackage{mathrsfs} 
\usepackage{tikz}
\usetikzlibrary{arrows, calc, patterns}

\usetikzlibrary{arrows.meta}
\usetikzlibrary{positioning}
\usetikzlibrary{bending}
\tikzset{>=triangle 45}
\allowdisplaybreaks[1]

\def\b{\beta}

\def\z{\zeta}

\def\S{\footnotesize}
\newcommand{\R}{{\mathbb R}}

\newcommand{\C}{{\mathbb C}}
\newcommand{\D}{{\mathbb D}}

\makeatletter
\newcommand{\inlineitem}[1][]{%
\ifnum\enit@type=\tw@
    {\descriptionlabel{#1}}
  \hspace{\labelsep}
\else
  \ifnum\enit@type=\z@
       \refstepcounter{\@listctr}\fi
    \quad\@itemlabel\hspace{\labelsep}
\fi}
\makeatother

\def\Cs{Ces\`{a}ro}

\DeclareMathOperator{\re}{Re}

\theoremstyle{plain}
\newtheorem{theorem}{Theorem}
\newtheorem{corollary}{Corollary}
\newtheorem{lemma}{Lemma}

\theoremstyle{definition}

\theoremstyle{remark}
\newtheorem{remark}{Remark}
\title{On the {\Cs}ro operator on the Hardy space in the upper half-plane}
\author{Valentin~V.~Andreev,\\
{\S Department of Mathematics}, \\
{\S Lamar University, Beaumont, TX 77710, USA}\\
\and
 Miron~B.~Bekker\footnote{Corresponding author},\\
 {\S Department of Mathematics}\\
{\S University of Pittsburgh at Johnstown, Johnstown, PA 15904, USA}\\
\and
Joseph~A.~Cima,
{\S Department of Mathematics},\\
{\S The University of North Carolina at Chapel Hill}\\
{\S  CB 3250, 329 Phillips Hall, Chapel Hill, NC, 27599, USA}
\and
{\S {\tt vvandreev8@gmail.com}, {\tt bekker@pitt.edu}, {\tt cima@email.unc.edu}}
}
\date{}
\begin{document}
\maketitle
\begin{abstract}
We discuss the {\Cs} operator on the Hardy space in the upper half-plane. We provide a new simple proof of the boundedness of this operator,  prove that this operator is equal to the sum of the identity operator and a unitary operator,
which implies its normality.
\end{abstract}
{\bf Keywords:} {Ces\`{a}ro operator, Hardy space, reproducing kernel, unitary \\operator.}\\
{\bf MSC 2020:} {47B38, 30H10, 47B32}

\section{Introduction}
In a classical paper by A.~Brown, P.~R.~Halmos and A.~L.~Shields \cite{BHS} were investigated the Ces\'{a}ro operators on different spaces ($l^2$, $L^2(0,1)$ $L^2(0,\infty)$). In particular, in \cite{BHS} was proved that $I-C_{\infty}^*$ is a bilateral shift of multiplicity one, where $C_{\infty}$ is the {\Cs} operator on $L^2(0,\infty)$. 
From that time study of {\Cs} operators attracted a lot of attention and many important results were obtained. We mention the result of  T.~L.~Kriete and David Trutt \cite{KT} who proved that the {\Cs} operator on $l^2$ is subnormal. 

Study of the {\Cs} operators on  spaces of analytic functions is currently an active field of research on intersection of operator theory and complex analysis. Many aspects of of the theory of {\Cs} operators in different settings, as well as some generalizations  of {\Cs} operators, are presented in the survey article by W.~Ross \cite{Ros}.
 
In 2013 A.~Arvanitidis and A.~Siskasis \cite{AS} introduced the {\Cs} operator on the Hardy spaces $H^p(\C_+)$ in the upper half-plane $\C_+$. They proved that the operator is bounded for $1<p\le\infty$ and determined its spectrum for $1<p<\infty$. These results were obtained using profound theorems from the theory of semigroups, in particular, the semigroups of composition operators. Application of this technique to the investigation of invariant subspaces of the {\Cs} operator on the Hardy space $H^2(\D)$ can be found in the recent paper \cite{GPR}.

In the present paper we give an alternative elementary proof of the boundedness of the {\Cs} operator on $H^p(\C_+)$. For $p=2$ we proof in an elementary way that the {\Cs} operator is the sum of the identity operator and a unitary operator, from which it follows, in particular, that the {\Cs} operator on $H^2(\C_+)$ is normal. In our approach  we use the fact that $H^2(\C_+)$ is a reproducing kernel Hilbert space. Note that the representation of the {\Cs} operator as the sum of the identity operator and a unitary operator can be obtained by further refinement of the methods that were used in \cite{AS}.

\section{Boundedness of the {\Cs} operator}
We denote by $\C_+$ the upper half-plane, $\C_+=\{z=x+iy:y>0\}$. For $1\le p\le\infty$ we denote by $H^p(\C_+)$ the Hardy space in $\C_+$, i.e. a function $f\in H^p(\C_+)$ ($1\le p<\infty$)  if $f$ is holomorphic in $\C_+$ and 
\begin{equation}\label{Hardy_Space_Definition}
\sup_{y>0}\int\limits_{-\infty}^{\infty}|f(x+iy)|^pdx<\infty.
\end{equation}
Also,
\begin{equation}\label{Hp_Norm}
\Vert f\Vert_p=\left[\sup_{y>0}\int\limits_{-\infty}^{\infty}|f(x+iy)|^pdx\right]^{1/p},
\end{equation}
and for $p=\infty$, $\Vert f\Vert_{\infty}=\sup_{z\in\C_+}|f(z)|$.
For $p=2$ the Hardy space $H^2(\C_+)$ is a Hilbert space and for $f,g\in H^2(\C_+)$ the inner product $<f,g>$ is defined by the formula
\begin{equation}\label{inner_product}
<f,g>=\int\limits_{-\infty}^{\infty}f(x)\overline{g(x)}dx
\end{equation}
where $f(x)=\lim_{z\to x}f(z)$ and the limit exists non-tangentially for almost all $x$ with respect to the Lebesgue measure.
The boundary function $f(x)$ belongs to $L^2(-\infty,\infty)$.

The {\Cs} operator $C$ is defined on $H^p(\C_+)$ by the following expression
\begin{equation}\label{Cesaro_Operator_Definition}
(Cf)(z)=\frac{1}{z}\int\limits_0^zf(\z)d\z.
\end{equation}
\begin{theorem}\label{boundedness_Cesaro_Operator}
For $1<p\le\infty$ the {\Cs} operator is bounded on $H^p(\C_+)$.
\end{theorem}
\begin{proof} The statement about $H^{\infty}(\C_+)$ is obvious, and $\Vert C\Vert_{\infty}\le 1$. Here 
$\Vert C\Vert_{\infty}$ is the norm of the {\Cs} operator on $H^{\infty}(\C_+)$.  

Consider now the case $1<p<\infty$. Put in \eqref{Cesaro_Operator_Definition} $\z=zs$, $0\le s\le 1$. Then \eqref{Cesaro_Operator_Definition} takes the form
\begin{equation*}
(Cf)(z)=\int\limits_0^1f(zs)ds.
\end{equation*}
Consequently,
\begin{gather*}
|Cf(z)|\le\int\limits_0^1|f(zs)|s^{\beta}s^{-\beta}ds\le \\
\left[\int\limits_0^1s^{\beta p}|f(zs)|^pds\right]^{1/p}\left[\int\limits_0^1s^{-\beta q}ds\right]^{1/q}=\\
\frac{1}{(1-\beta q)^{1/q}}\left[\int\limits_0^1s^{\beta p}|f(zs)|^pds\right]^{1/p}.
\end{gather*}
Here $1/p+1/q=1$ and $0<\beta<1/q$. Therefore we have 
\begin{gather*}
\int\limits_{-\infty}^{\infty}|(Cf)(x+iy)|^pdx\le\frac{1}{(1-\beta q)^{p/q}}\int\limits_0^1\left[\int\limits_{-\infty}^{\infty}s^{\beta p}|f(sx+isy)|^pdx\right]ds= \\
\frac{1}{(1-\beta q)^{p/q}}\int\limits_0^1s^{\beta p-1}\left[\int\limits_{-\infty}^{\infty}|f(t+isy)|^pdt\right]ds.
\end{gather*}
Since for each $s\ge 0$ and $f\in H^p(\C_+)$
\begin{equation*}
 \int\limits_{-\infty}^{\infty}|f(t+isy)|^pdt\le\Vert f\Vert^p_p
\end{equation*}
we obtain
\begin{equation*}
\int\limits_{-\infty}^{\infty}|(Cf)(x+iy)|^pdx\le\frac{\Vert f\Vert^p_p}{(1-\beta q)^{p/q}}\int_0^1s^{\b p-1}ds=\frac{1}{(1-\beta q)^{p/q}\beta p}\Vert f\Vert_p^p.
\end{equation*}
Last inequality means
\begin{equation}
\Vert Cf\Vert_p\le\frac{1}{(1-\beta q)^{1/q}(\beta p)^{1/p}}\Vert f\Vert_p,
\end{equation}
that is the {\Cs} operator is bounded. Moreover, $\Vert C\Vert\le \inf\{(1-\beta q)^{-1/q}(\beta p)^{-1/p}:0<\beta<(p-1)/p\}.$
A simple calculation gives 
\begin{equation*}
\Vert C\Vert_p\le\frac{p}{p-1}.
\end{equation*}
In \cite{AS} it was proved that $\Vert C\Vert_p=p/(p-1)$. \\
Simple examples show the theorem does not hold for $p=1$.
\end{proof}
\begin{remark}
In \cite{Sed} it was proved that the class $H^p(\C_+)$ coincides with the class of functions $f$ which are holomorphic in the upper half-plane and such that 
\begin{equation*}
(\Vert f\Vert_p^*)^p=\sup\limits_{0<t<\pi}\int\limits_0^{\infty}|f(re^{it})|^pdr<\infty,\quad 0<p<\infty
\end{equation*}
and $\Vert f\Vert_p^*$ is equivalent to the usual $H^p$ norm, that is there are constants $0<A_p\le B_p<\infty$ which depend only on $p$ such that 
$A_p\Vert f\Vert_p\le\Vert f\Vert_p^*\le B_p\Vert f\Vert_p$ (for $p=2$ this class of functions was considered in \cite{Dzr}).
This fact can also be used to prove the boundedness of the {\Cs} operator on $H^p(\C_+)$ in the same way as above. 
\end{remark}
\begin{remark}
The same arguments that were used in the proof of Theorem \ref{boundedness_Cesaro_Operator} allow to prove the following slightly more general statement:\\
{\it Let $\re\mu>-\dfrac{p-1}{p}$, $1<p<\infty$. Then the operator $C_{\mu}$ defined by 
\begin{equation*}
(C_{\mu}f)(z)=\frac{1}{z^{\mu+1}}\int_0^z\z^{\mu}f(\z)d\z
\end{equation*}
is bounded on $H^p(\C_+)$.}
\end{remark}
\vskip 0.5truecm
\section{{\Cs} operator on $H^2(\C_+)$}
In this section we  assume that $p=2$ and use the Hilbert space structure of $H^2(\C_+)$, in particular, the fact that $H^2(\C_+)$ is a reproducing kernel Hilbert space.

\begin{lemma}\label{adjoint}The operator $C^*$ on $H^2(\C_+)$ is defined by the formula
\begin{equation}\label{adjoint_operator}
(C^*f)(z)=\frac{1}{2\pi i}\int\limits_{-\infty}^{\infty} f(t)\frac{1}{t}\log{\left(1-\frac{t}{z}\right)}dt
\end{equation}
\end{lemma}
\begin{proof} Recall that for any $f\in H^2(\C_+)$ the boundary value $\lim_{y\downarrow 0}f(x+iy)$ exists for almost all $x$ 
and the boundary function belongs to $L^2(\R)$. We use the same letter $f$ for the boundary function. The values of function $f$ in the upper half-plane are recovered from the boundary function by the Cauchy formula
\begin{equation}\label{Cauchy_integral}
f(z)=\frac{1}{2\pi i}\int\limits_{\infty}^{\infty}\frac{f(t)}{t-z}dt.
\end{equation} 
(see, for example \cite {Dur1}, Theorem 11.8). 
Formula \eqref{Cauchy_integral} means that the function
\begin{equation}\label{RK}
 k_z(w)=\frac{1}{2\pi i}\frac{1}{\bar{z}-w}
\end{equation}
is the reproducing kernel for $H^2(\C_+)$, that is 
$$
f(z)=<f,k_z>.
$$
Observe that 
\begin{equation}\label{C_on_RK}
 (Ck_z)(\z)=\frac{1}{2\pi i\z}\int_0^{\z}\frac{1}{\bar{z}-w}dw=
 -\frac{1}{2\pi i\z}\log{\left(1-\frac{\z}{\bar{z}}\right)}
\end{equation}
Since 
$$
(C^*f)(z)=<C^*f,k_z>=<f,Ck_z>
$$
formula \eqref{adjoint_operator} follows.
\end{proof}

Later on we will need a special case of formula \eqref{adjoint_operator}.\\ Because $(C^*k_z)(\z)=<C^*k_z,k_{\z}>=\overline{<Ck_{\z},k_z>}$ it follows that 
\begin{equation}\label{C*_on_RK}
(C^*k_z)(\z)=\frac{1}{2\pi i}\frac{1}{\bar{z}}\log{\left(1-\frac{\bar{z}}{\z}\right)}.
\end{equation}
\begin{theorem}\label{I+U}
For the {\Cs} operator on $H^2(\C+)$ the following equality is valid
\begin{equation}\label{representation}
C=I+U
\end{equation}
where $I$ is the identity operator and $U$ is a unitary operator. In particular, the {\Cs} operator on $H^2(\C_+)$ is normal.
\end{theorem}
\begin{proof} We claim that for any $s,t>0$ the following equalities are valid:
\begin{equation}\label{C-I}
<(C-I)k_{is},(C-I)k_{it}>=<k_{is},k_{it}>
\end{equation}
\begin{equation}\label{C*-I}
<(C^*-I)k_{is},(C^*-I)k_{it}>=<k_{is},k_{it}>
\end{equation}
Indeed,
\begin{gather*}
<(C-I)k_{is},(C-I)k_{it}>=\\
<Ck_{is},Ck_{it}>-<Ck_{is},k_{it}>-<k_{is},Ck_{it}>+<k_{is},k_{it}>=\\
<Ck_{is},Ck_{it}>-<Ck_{is},k_{it}>-<C^*k_{is},k_{it}>+<k_{is},k_{it}>.
\end{gather*}
Consider the inner product $<Ck_{is},Ck_{it}>$. Using \eqref{C_on_RK} and integrating by parts on the interval of the form 
$[-n,n]$ and letting $n$ tend to infinity one obtains
\begin{gather*}
<Ck_{is},Ck_{it}>=\frac{1}{4\pi^2}\int\limits_{-\infty}^{\infty}\frac{1}{x^2}\log{\left(1-i\frac{x}{s}\right)}\log{\left(1+i\frac{x}{t}\right)}dx=\\
\frac{1}{4\pi^2}\left\{-\frac{1}{x}\log{\left(1-i\frac{x}{s}\right)}\log{\left(1+i\frac{x}{t}\right)}|_{x=-\infty}^{\infty}+
\int\limits_{-\infty}^{\infty}\frac{1}{x}\frac{d}{dx}\left[\log{\left(1-i\frac{x}{s}\right)}\log{\left(1+i\frac{x}{t}\right)}\right]dx\right\}=\\
\frac{1}{4\pi^2}\int\limits_{-\infty}^{\infty}\left[\frac{1}{x}\frac{\log{\left(1+i\dfrac{x}{t}\right)}}{1-i\dfrac{x}{s}}
\left(-\frac{i}{s}\right)+\frac{1}{x}\frac{\log{\left(1-i\dfrac{x}{s}\right)}}{1+i\dfrac{x}{t}}
\left(\frac{i}{t}\right)
\right]dx=\\
\frac{1}{4\pi^2}\int\limits_{-\infty}^{\infty}\frac{1}{x}\frac{\log{\left(1+i\dfrac{x}{t}\right)}}{x+is}dx+
\frac{1}{4\pi^2}\int\limits_{-\infty}^{\infty}\frac{1}{x}\frac{\log{\left(1-i\dfrac{x}{s}\right)}}{x-it}dx
\end{gather*}
Note that for both integrands the point $x=0$ is a regular point. Both integrals are easily evaluated using the calculus of residue. For the first integral we need to close the contour of integration in the lower half plane and for the second integral we need to close the contour of integration in the upper half plane. As result one obtains
\begin{equation}\label{CkCk}
<Ck_{is},Ck_{it}>=\frac{1}{2\pi}\left[\frac{1}{s}\log{\left(1+\frac{s}{t}\right)}+\frac{1}{t}\log{\left(1+\frac{t}{s}\right)}
\right]
\end{equation}
From \eqref{CkCk} it follows that $\Vert Ck_{is}\Vert^2=\dfrac{\log{2}}{\pi s}$.

According to \eqref{C_on_RK} and \eqref{C*_on_RK}
\begin{gather*}
<Ck_{is},k_{it}>=(Ck_{is})(it)=\frac{1}{2\pi t}\log{\left(1+\frac{t}{s}\right)},\\
<C^*k_{is},k_{it}>=(C^*k_{is})(it)=\frac{1}{2\pi s}\log{\left(1+\frac{s}{t}\right)}.
\end{gather*}
Now \eqref{C-I} follows.

In order to prove \eqref{C*-I} we evaluate $<C^*k_{is},C^*k_{it}>$. Using \eqref{C*_on_RK} one obtains
\begin{equation*}
<C^*k_{is},C^*k_{it}>=\frac{1}{4\pi^2st}\int\limits_{-\infty}^{\infty}\log{\left(1+\frac{is}{x}\right)}
\log{\left(1-\frac{it}{x}\right)}dx.
\end{equation*}
Making change of variables $x\mapsto 1/x$ and then integrating by parts one obtains
\begin{gather*}
<C^*k_{is},C^*k_{it}>=\frac{1}{4\pi^2st}\int\limits_{-\infty}^{\infty}\frac{1}{x^2}\log{(1+isx)}\log{(1-itx)}dx=\\
\frac{1}{4\pi^2st}\int\limits_{-\infty}^{\infty}\frac{is\log{(1-itx)}}{x(1+isx)}dx+\frac{1}{4\pi^2st}\int\limits_{-\infty}^{\infty}\frac{-it\log{(1+isx)}}{x(1-itx)}dx=\\
\frac{1}{4\pi^2st}\int\limits_{-\infty}^{\infty}\frac{\log{(1-itx)}}{x(x-\dfrac{i}{s})}dx+
\frac{1}{4\pi^2st}\int\limits_{-\infty}^{\infty}\frac{\log{(1+isx)}}{x(x+\dfrac{i}{t})}dx.
\end{gather*}
Using calculus of residue again as above one obtains 
\begin{equation*}
<C^*k_{is},C^*k_{it}>=\frac{1}{2\pi}\left[\frac{1}{t}\log{\left(1+\frac{t}{s}\right)}+\frac{1}{s}\log{\left(1+\frac{s}{t}\right)}\right].
\end{equation*}
In particular, we have proved that for any $s,t>0$ 
\begin{equation}\label{main_equality}
<Ck_{is},Ck_{it}>=<C^*k_{is},C^*k_{it}>.
\end{equation}
Using again \eqref{C_on_RK} and \eqref{C*_on_RK} we obtain equality \eqref{C*-I}.
Therefore, for any vector $f\in H^2(\C_+)$ of the form $f=\sum a_jk_{is_j}$ (the sum is finite)
we have 
\begin{equation*}
\Vert (C-I)f\Vert^2=\Vert (C^*-I)f\Vert^2=\Vert f\Vert^2.
\end{equation*}
Let $\{s_j\}_{j=1}^{\infty}$ be a set of positive distinct numbers which has a limit point $s_0>0$. Then the set of reproducing 
 kernels $\{k_{is_j}\}$ is linearly independent and $\vee\{k_{is_j}:j=1,2,\ldots\}=H^2(\C_+)$. Indeed, if $h\in H^2(\C_+)$ and 
 $h\perp k_{is_j}$ for all $j=1,2,\ldots$, then $f(is_j)=0$ and, since $f$ is holomorphic, $f=0$. Our previous consideration implies that $\Vert (C-I)f\Vert^2=\Vert (C^*-I)f\Vert^2=\Vert f\Vert^2$ for any $f\in H^2(\C_+)$, that is 
that is the operator $C-I$ is unitary. 
\end{proof}

\begin{corollary}
The operator $C+C^*$ is a positive self-adjoint operator.
\end{corollary}
Since the operator $C-I$ is unitary it follows that $(C^*-I)(C-I)=I$, that is $C^*C=C+C^*$.
\vskip 0.5truecm 
The next corollary  is a byproduct of the proof of Theorem \ref{I+U}.
\begin{corollary}
For any $s,t>0$ 
\begin{equation*}
\frac{1}{2}\left[\frac{1}{s}\log{\left(1+\frac{s}{t}\right)}+\frac{1}{t}\log{\left(1+\frac{t}{s}\right)}\right]\le\frac{\log{2}}{\sqrt{st}}
\end{equation*}
with equality if and only if $s=t$.
\end{corollary}
\vskip 1.0truecm
{\bf Acknowledgement}. Authors are very thankful to Dr. Mark~A.~Nudel'man (Odessa, Ukraine) for numerous useful discussions and comments.

\end{document}